\documentclass[letterpaper, 10 pt, conference]{ieeeconf}  

\IEEEoverridecommandlockouts                              
\usepackage{graphicx}
\usepackage{amsmath}
\usepackage{subfig}
\usepackage{psfrag}
\usepackage{amssymb}
\usepackage{latexsym}
\usepackage{subfig}
\usepackage{algorithm}
\usepackage{algorithmic}
\usepackage{color}
\usepackage{enumerate}
\usepackage{cite}
 \usepackage{verbatim}

\newtheorem{remark}{\bf Remark}[section]
\newtheorem{proposition}{\bf Proposition}[section]

\newtheorem{definition}{\bf Definition}[section]
\newtheorem{example}{\bf Example}[section]
\newtheorem{theorem}{\bf Theorem}[section]
\newtheorem{problem}{\bf Problem}[section]

\renewcommand\vec[1]{\ensuremath\boldsymbol{#1}}
\newcommand{\ssim}{\ensuremath{\slash_{\sim}}}
\newcommand{\cS}{\mathcal{S}}
\newcommand{\cP}{\mathcal{P}}

\newcommand{\interior}{\mathrm{int}}
\newcommand{\bx}{\textbf{x}}
\newcommand{\bo}{\textbf{o}}
\newcommand{\bq}{\textbf{q}}
\newcommand{\bp}{\textbf{p}}

\newcommand{\pre}{\mathrm{Pre}}
\newcommand{\eq}{\mathrm{eq}}
\newcommand{\mx}{\mathcal{X}}
\newcommand{\mr}{\mathcal{R}}
\newcommand{\mt}{\mathcal{D}}
\newcommand{\pset}{\tilde\cP}
\newcommand{\tcp}{\tilde\cP}

\newcommand{\Qsat}{{\mathcal{X}^S}}
\newcommand{\FSA}{\mathcal{A}}
\newcommand{\Prod}{\mathcal{PA}}

\newcommand{\setI}{\Sigma}

\newcommand{\Act}{{\ensuremath{\Sigma}}}
\newcommand{\toActG}{{\stackrel{\sigma}{\to}}}
\newcommand{\toAct}{{\stackrel{\sigma}{\to}_{e}}}
\newcommand{\toActQ}{{\stackrel{\sigma}{\to}}}
\newcommand{\bsigma}{\ensuremath{\vec{\sigma}}}
\newcommand{\TS}{\mathcal{T}}

\newcommand{\LTLand}{\wedge}
\newcommand{\LTLor}{\vee}
\newcommand{\LTLNext}{\mathsf{X}\, }

\newcommand{\LTLEvent}{\mathsf{F} \,}
\newcommand{\LTLUntil}{\, \mathsf{U} \,}

\newcommand{\Rset}{\mathbb{R}}
\newcommand{\Zset}{\mathbb {Z}}
\newcommand{\rn}{\mathbb{R}^{n}}

\newcommand{\ie}{{\it i.e., }}
\newcommand{\eg}{{\it e.g., }}

\newcommand{\be}{\begin{equation}}
\newcommand{\ee}{\end{equation}}

\newcommand{\st}{\,|\,}

\newcommand{\computepre}{\mathtt{ComputePre}}
\newcommand*{\myscale}{0.33}
\newcommand*{\myscaleex}{0.30}

\title{\LARGE \bf
Finite Bisimulations for Switched Linear Systems
}

\author{Ebru Aydin Gol, Xuchu Ding, Mircea Lazar and Calin Belta
\thanks{This work was partially supported by the NSF under grants CNS-0834260 and CNS-1035588 and by the ONR under grant MURI N00014-09-1051 at Boston University, and by Veni grant 10230 at Eindhoven University of Technology.}
\thanks{Ebru Aydin Gol (ebru@bu.edu) and Calin Belta (cbelta@bu.edu) are with Boston University. Xuchu Ding (dingx@utrc.utc.com) is with United Technologies Research Center. Mircea Lazar (m.lazar@tue.nl) is with Eindhoven University of Technology.}
}
\usepackage{hyperref}

\begin{document}

\maketitle
\thispagestyle{empty}
\pagestyle{empty}

\begin{abstract}
In this paper, we consider the problem of constructing a finite bisimulation quotient for a discrete-time switched linear system in a bounded subset of its state space. Given a set of observations over polytopic subsets of the state space and a switched linear system with stable subsystems, the proposed algorithm generates the bisimulation quotient in a finite number of steps with the aid of sublevel sets of a polyhedral Lyapunov function. Starting from a sublevel set that includes the origin in its interior, the proposed algorithm iteratively constructs the bisimulation quotient for any larger sublevel set. The bisimulation quotient can then be further used for synthesis of the switching law and system verification with respect to specifications given as syntactically co-safe Linear Temporal Logic formulas over the observed polytopic subsets.
\end{abstract}

\section{INTRODUCTION}\label{sec:intro}

In recent years, there has been a trend to bridge the gap between control theory and formal methods. Control theory allows for analysis and control of ``complex'' dynamical systems with infinite state spaces, such as systems of controlled differential equations,
against ``simple'' specifications, such as stability and reachability. In formal methods, ``simple'' systems, such as finite transition systems, are checked against ``complex'' (rich and expressive) specification languages, such as temporal logics. Recent studies show that certain classes of dynamical systems can be abstracted to finite transition systems. Applications in robotics \cite{belta2007symbolic}, multi-agent control systems \cite{Loizou04}, and bioinformatics \cite{batt2005validation} show that model checking and automata games can be used to analyze and control systems with non-trivial dynamics from specifications given as temporal logic formulas.

In this paper, we focus on switched linear systems made of stable subsystems, and show that a finite bisimulation abstraction of the system can be efficiently constructed within some relevant, bounded subset of the state space. Since the bisimulation quotient preserves all properties that are expressible in frameworks as rich as $\mu$-calculus, and implicitly Computation Tree Logic (CTL)  and Linear Temporal Logic (LTL) (see \eg \cite{Milner89,browne1988characterizing,davoren2000logics}), it can be readily used for system verification and controller synthesis against such specifications. We show how our method can be used for both controller synthesis and verification from specifications given as arbitrary formulas of a fragment of LTL, called syntactically co-safe LTL (scLTL). For controller synthesis, we find the largest set of initial states and switching sequences such that all system trajectories satisfy a given formula. For verification, we find the largest set of initial states such that all system trajectories satisfy the formula under arbitrary switching.

The concept of constructing a finite quotient of an infinite system has been widely studied, \eg  \cite{Tabuada06_LTL, Krogh:2001,alur1994theory}.  It is known that finite state bisimulation quotients exist only for specific classes of systems (\eg timed automata \cite{alur1994theory} and controllable linear systems \cite{Tabuada06_LTL}), and the well known bisimulation algorithm \cite{Milner89}
in general does not terminate \cite{YoBe-TAC-2010}.  Approximately bisimilar finite abstractions for continuous-time switched systems were constructed under incremental stability assumptions  in \cite{Girard:2010tac}.  For piecewise linear systems, guided refinement procedures were employed with the goal of
constructing the quotient system for verification of certain properties \cite{YoBe-TAC-2010, Krogh:2001}.

We propose to obtain a finite bisimulation quotient of the system in a computationally feasible manner by only considering the system behavior within a relevant state space that does not contain the origin,  \ie in between two positively invariant compact sets that contain the origin. Our approach relies upon the existence of a \emph{polyhedral} Lyapunov function, which is a necessary condition for stability under arbitrary switching, see, e.g., \cite{lazar2010infinity}. We propose to partition the state space by using sublevel sets of the Lyapunov function.  Such sublevel sets, which are polytopic, allow us to generate the bisimulation quotient incrementally as the abstraction algorithm iterates, with no ``holes'' in the covered state space.  Since we can obtain polytopic sublevel sets of any size from the Lyapunov function, the balance between the size of the abstracted state space and the amount of computation can be easily adjusted and controlled. Starting from the observation that
the existence of the Lyapunov function renders the origin asymptotically stable for the switched system, its trajectories can only spend a finite time in the region of interest. As a result, we restrict our attention to LTL specifications that can be satisfied in finite time, such as scLTL formulas.

This paper is a natural, but non-trivial extension of our recent work \cite{ADHS2012_LinearSysBisim}, in which we used polytopic sublevel sets to generate a bisimulation quotient for a discrete autonomous linear system.  Another conceptually related work is \cite{sloth2011abstraction}, where $n$ Lyapunov functions were used for the abstraction of $n$-dimensional continuous-time Morse-Smale systems (\eg hyperbolic linear systems) to timed automata.  The abstraction proposed therein is weaker than bisimulation, but it can be used to verify safety properties.
While both \cite{sloth2011abstraction} and this work use sublevel sets for abstraction, the main difference between \cite{sloth2011abstraction} and this approach comes from the usage of \emph{polyhedral Lyapunov functions}, and therefore different classes of systems for which the methods apply. Our approach removes the need for multiple orthogonal Lyapunov functions, and we argue that it allows for a more tractable implementation since the abstraction of timed automata is expensive by itself \cite{alur1994theory},  and polytopic sublevel sets ensure that the abstraction algorithm requires only polytopic operations.

The rest of the paper is organized as follows.  We introduce preliminaries in Sec.~\ref{sec:prelim} and formulate the problem in Sec.~\ref{sec:problem}.  We present the algorithm to generate the bisimulation quotient in Sec.~\ref{sec:genBis}, and we show in Sec.~\ref{sec:verification} how the resulting bisimulation quotient can be used to synthesize switching control laws and verify the system behavior against temporal logic formulas. Conclusions are summarized  in Sec.~\ref{sec:concl}.

\section{PRELIMINARIES}\label{sec:prelim}

For a set $\cS$, $\interior(\cS)$, $|\cS|$, and $2^\cS$ stand for its interior, cardinality, and power set, respectively. For $\lambda\in\Rset$ and $\cS\subseteq\Rset^n$, let $\lambda \cS:=\{\lambda x\mid x\in \cS\}$. We use $\Rset,\:\Rset_+,\:\Zset,$ and $\Zset_+$ to denote the sets of real numbers, non-negative reals, integer numbers, and non-negative integers. For $m,n\in \Zset_+$, we use $\Rset^n$ and $\Rset^{m\times n}$ to denote the set of column vectors and matrices with $n$ and $m\times n$ real entries.
For a vector $v$ or a matrix $A$, we denote $v^\top$ or $A^{\top}$ as its transpose, respectively.

For a vector $x\in\Rset^n$, $[x]_i$ denotes the $i$-th element of $x$ and $\|x\|_\infty=\max_{i=1,\ldots,n}\left|[x]_i\right|$ denotes the infinity norm of $x$, where $|\cdot|$ denotes the absolute value.
For a matrix $Z\in\Rset^{l\times n}$, let $\|Z\|_{\infty}:=\sup_{x\neq0}\frac{\|Zx\|_{\infty}}{\|x\|_{\infty}}$ denote its induced matrix infinity norm.

A $n$-dimensional \emph{polytope} $\cP$ (see, \eg \cite{grunbaum2003convex}) in $\rn$ can be described as the convex hull of $n+1$ affinely independent points in $\rn$.  Alternatively, $\cP$ can be described as the intersection of $k$, where $k\geq n+1$, closed half spaces, \ie there exists $k\geq n+1$ and $H_{\cP}\in \Rset^{k\times n}$, $h_{\cP}\in \Rset^{k}$, such that
\begin{equation}
\label{eq:definitionofpolytope}
\cP=\{x\in \rn \mid H_{\cP}x\leq h_{\cP}\}.
\end{equation}

We assume polytopes in $\rn$ are $n$-dimensional unless noted otherwise.
The set of boundaries of a polytope $\cP$ are called \emph{facets}, denoted by $f(\cP)$, which are themselves $(n-1)$-dimensional polytopes.  A \emph{semi-linear} set (also called a \emph{polyhedron} in literature) in $\rn$ is defined as finite unions, intersections and complements of sets
$\{x\in \rn \mid a^{\top}x\sim b, \sim \in \{=, <\}\}$, for some $a\in \rn$ and $b\in \Rset$.  Note that a convex and bounded semi-linear set is equivalent to a polytope with some of its facets removed.

\subsection{Transition systems and bisimulations}
\label{sec:tsandbisim}

\begin{definition}
\label{def:tran_sys}
A transition system (TS) is a tuple $\TS=(Q,\Act, \to,\Pi,h)$, where
\begin{itemize}
\item $Q$ is a (possibly infinite) set of states;
\item $\Act$ is a set of inputs;
\item $\to \subseteq Q\times \Act \times Q$ is a set of transitions;
\item $\Pi$ is a finite set of observations; and
\item $h : Q \longrightarrow 2^{\Pi}$ is an observation map.
\end{itemize}
We denote $x\toActG x'$ if $(x,\sigma,x')\in\to$.   We assume $\TS$ to be non-blocking, \ie for each $x\in Q$, there exists $x'\in Q$ and $\sigma \in \Act$ such that $x \toActG x'$. An {\em input word} is defined as an infinite sequence $\bsigma=\sigma _0\sigma_1\ldots$ where $\sigma_k \in \Sigma$ for all $k\in \Zset_{+}$. A {\em trajectory} of $\TS$ produced by an input word $\bsigma=\sigma _0\sigma_1\ldots$ and originating at state $x_{0}$ is an infinite sequence $\bx=x_{0}x_{1}...$ where $x_{k}\stackrel{\sigma_k}{\to} x_{k+1}$ for all $k\in \Zset_{+}$. A trajectory $\bx$ generates a word $\bo=o_{0}o_{1}...$, where $o_{k}=h(x_{k})$ for all $k\in \Zset_{+}$.
\end{definition}

The TS $\TS$ is \emph{finite} if $|Q|<\infty$ and $|\Sigma|<\infty$, otherwise $\TS$ is \emph{infinite}.  Moreover, $\TS$ is \emph{deterministic} if $x \toActG x'$ implies that there does not exist $x''\neq x'$ such that $x \toActG x''$; otherwise, $\TS$ is called \emph{non-deterministic}. Given a set $X\subseteq Q$, we define the set of states $\pre_\TS (X, \sigma)$ that reach $X$ in one step when input $\sigma$ is applied as
\begin{eqnarray}
\label{eq:pre}
&\pre_\TS (X, \sigma):=\{x\in Q \mid \exists x'\in X, x \toActG x'\}.\label{eq:presigma}
\end{eqnarray}

States of a TS can be related by a relation $\sim\subseteq Q\times Q$.  For convenience of notation, we denote $x\sim x'$ if $(x,x')\in \sim$.   The subset $X\subseteq Q$ is called an \emph{equivalence class} if $x,x'\in X \Leftrightarrow x\sim x'$.  We denote by $Q/_{\sim}$ the set labeling all equivalence classes and define a map $\eq:Q/_{\sim}\mapsto 2^Q$ such that $\eq(X)$ is the set of states in the equivalence class $X\in Q\ssim$.
\begin{definition}
\label{def:opr}
We say that a relation $\sim$ is \emph{observation preserving} if for any $x,x'\in Q$, $x\sim x'$ implies that $h(x)=h(x')$.
\end{definition}

\begin{definition}
A finite \emph{partition} $P$ of a set $\cS$ is a finite collection of sets $P:=\{P_{i}\}_{i\in I}$, such that $\cup_{i\in I}P_{i}=\cS$ and $P_{i}\cap P_{j}=\emptyset$ if $i\neq j$.  A finite \emph{refinement} of $P$ is a finite partition $P'$ of $\cS$ such that for each $P_{i}\in P'$, there exists $P_{j}\in P$ such that $P_{i}\subseteq P_{j}$. 
\end{definition}

A partition naturally induces a relation, and an observation preserving relation induces a quotient TS.
One can immediately verify that a refinement of an observation preserving partition is also observation preserving.
\begin{definition}
An observation preserving relation $\sim$ of a TS $\TS=(Q,\Act,\to,\Pi,h)$ induces a \emph{quotient transition system} $\TS\ssim=(Q\ssim, \Act, \to_{\sim},\Pi,h_{\sim})$, where $Q/_{\sim}$ is the set labeling all equivalence classes. 
The transitions of $\TS\ssim$ are defined as $X \toActG_{\sim}Y$ if and only if there exists $x\in \eq(X)$ and $x'\in \eq(Y)$ such that $x \toActG x'$. The observation map is defined as $h_{\sim}(X):=h(x)$, where $x\in \eq(X)$.
\end{definition}

\begin{definition}\label{def:bisimulation}
 Given a TS $\TS=(Q,\Act, \to,\Pi,h)$, a relation $\sim$ is a bisimulation relation of $\TS$ if (1) $\sim$ is observation preserving; and (2) for any $x_{1},x_{2}\in Q, \sigma \in \Sigma$, if $x_{1}\sim x_{2}$ and $x_{1}\toActG x_{1}'$, then there exists $x_{2}'\in Q$ such that $x_{2}\toActG x_{2}'$ and $x_{1}'\sim x_{2}'$.
\end{definition}

If $\sim$ is a bisimulation, then the quotient transition system $\TS\ssim$ is called a \emph{bisimulation quotient} of $\TS$.  In this case, $\TS$ and $\TS\ssim$ are said to be \emph{bisimilar}.  Bisimulation is a very strong equivalence relation between systems. In fact, it preserves properties expressed in temporal logics such as LTL, CTL and $\mu$-calculus  \cite{Milner89,browne1988characterizing,davoren2000logics}.  As such, it is used as an important tool to reduce the complexity of system verification or controller synthesis, since the bisimulation quotient (which may be finite) can be verified or used for controller synthesis instead of the original system.

\subsection{Polyhedral Lyapunov functions}
\label{sec:sec:lyapunovfun}

Consider an autonomous discrete-time system,
\begin{equation}
\label{eq:dynDisSys}
x_{k+1}=\Phi(x_{k}),\quad k\in \Zset_{+},
\end{equation}
where $x_{k}\in \rn$ is the state at the discrete-time instant $k$ and $\Phi: \rn\mapsto \rn$ is an arbitrary map with $\Phi(0)=0$.  Given a state $x\in\rn$, $x':=\Phi(x)$ is called a \emph{successor} state of $x$.

\begin{definition}
\label{defPI}
Let $\lambda\in[0,1]$. We call a set $\cP\subseteq\Rset^n$ \emph{$\lambda$-contractive} (shortly, contractive) if for all $x\in\cP$ it holds that $\Phi(x)\in\lambda\cP$. For $\lambda=1$, we call $\cP$ a \emph{positively invariant} set.
\end{definition}

\begin{theorem}
\label{def:lf}
Let $\mx$ be a positively invariant set for \eqref{eq:dynDisSys} with $0\in \interior(\mx)$.  Furthermore, let $\alpha_{1},\alpha_{2}\in \mathcal{K}_{\infty}$, $\rho\in(0,1)$ and $V:\rn\mapsto \Rset_{+}$ such that:
\begin{align}
\alpha_{1}(\|x\|)\leq V(x)&\leq\alpha_{2}(\|x\|), \forall x\in \mx, \label{eq:LFcond1} \\
V(\Phi(x))&\leq\rho V(x), \forall x\in \mx. \label{eq:LFcond2}
\end{align}
Then system \eqref{eq:dynDisSys} is asymptotically stable in $\mx$.
\end{theorem}

The proof of Thm.~\ref{def:lf} can be found in \cite{jiang2002converse, lazarThesis}. 

\begin{definition} A function $V:\rn\mapsto \Rset_{+}$ is called a \emph{Lyapunov function} (LF) in $\mx$ if it satisfies \eqref{eq:LFcond1} and \eqref{eq:LFcond2}. If $\mx=\rn$, then $V$ is called a \emph{global Lyapunov function}.
\end{definition}

The parameter $\rho$ is called the \emph{contraction rate} of $V$. For any $\Gamma>0$, $\cP_{\Gamma}:=\{ x\in \rn\st V(x)\leq \Gamma\}$ is called a \emph{sublevel set} of $V$.

For the remainder of this paper we consider LFs defined using the infinity norm, i.e.,
\be
\label{eq:polyLF}
V(x)=\|Lx\|_{\infty},\quad L\in \Rset^{l\times n}, l \geq n, l\in \Zset_{+},
\ee
where $L$ has full-column rank. Notice that infinity norm Lyapunov functions are a particular type of polyhedral Lyapunov functions. We opted for this type of function to simplify the exposition but in fact, the proposed abstraction method applies to general polyhedral Lyapunov functions defined by Minkowski (gauge) functions of polytopes in $\Rset^n$ with the origin in their interior.
\begin{proposition}
\label{prop:rhocontractive}
Suppose that $L\in \Rset^{l\times n}$ has full-column rank and $V$ as defined in \eqref{eq:polyLF} is a global LF for system \eqref{eq:dynDisSys} with contraction rate $\rho\in(0,1)$. Then for all $\Gamma>0$ it holds that $\cP_\Gamma$ is a polytope and $0\in \interior(\cP_{\Gamma})$. Moreover, if $\Phi(x)$ takes values arbitrarily from a set $\{Ax \st A \in {\cal A} \}$ for some polyhedral set ${\cal A} \subseteq  \Rset^{n\times n}$, then \emph{for all $\Gamma>0$} it holds that $\cP_\Gamma$ is a $\rho$-contractive polytope for \eqref{eq:dynDisSys}.
\end{proposition}

The proof of the above result is a straightforward application of results in \cite{blanchini:1994, lazar2010infinity}.

\section{PROBLEM FORMULATION}\label{sec:problem}

In this paper, we consider discrete-time switched linear systems, \ie
\be
\label{eq:linearDyn}
x_{k+1}=A_{\bsigma(k)}x_{k},\quad  \bsigma(k) \in \setI, k\in \Zset_{+},
\ee
where $\bsigma : \Zset_+ \to \setI$ is a switching sequence that selects the active subsystem from a finite index set $\setI$ and
$A_i \in\Rset^{n\times n}$ is a strictly stable (i.e., Schur) matrix for all $i \in \setI$.  We assume that a global polyhedral Lyapunov function (LF) of the form \eqref{eq:polyLF} with contraction rate $\rho\in(0,1)$ is known for system \eqref{eq:linearDyn}. 

Let $\mx$ be a polytope $\mx:=\{x\mid\|Lx\|_\infty\leq \Gamma_\mx\}$ and $\mt$ be a polytope $\mt:=\{x\mid\|Lx\|_\infty\leq \Gamma_\mt\}$, where $L$ corresponds to the polytopic LF \eqref{eq:polyLF} of system \eqref{eq:linearDyn} and we assume that $0<\Gamma_\mt<\Gamma_\mx$.  Note that $\mt\subset \mx$ and $0\in \interior(\mt)\subset \interior(\mx)$.  We call $\mx$ the working set and $\mt$ the target set.  We are interested in synthesis of control strategies and verification of the system behavior within $\mx$ with respect to polytopic regions in the state space, until the target set $\mt$ is reached (since $\mt$ is positively invariant, the system trajectory will be confined within $\mt$ after $\mt$ is reached).  

We assume that there exists a set $\mr$ of polytopes indexed by a finite set $R$, \ie $\mr:=\{\mr_{i}\}_{i\in R}$, where $\mr_{i}\subseteq \mx\setminus\mt$ for all $i\in R$, and $\mr_{i}\cap \mr_{j} =\emptyset$ for any $i\neq j$. The set $\mr$ represents regions of interest in the relevant state space, and the polytopes in $\mr$ are considered as observations of \eqref{eq:linearDyn}.  Therefore, informally, a trajectory of \eqref{eq:linearDyn} $x_{0}x_{1}\ldots$ produces an infinite sequence of observations $o_{0}o_{1}\ldots$, such that $o_{i}$ is the index of the polytope in $\mr$ visited by state $x_{i}$, or $o_{i}=\emptyset$ if $x_{i}$ is in none of the polytopes.

\begin{example}
\label{ex:simpleEx}
Consider a system as in~\eqref{eq:linearDyn}, $\Sigma=\{1,2\}$, $A_1=\begin{pmatrix}-0.65 & 0.32 \\ -0.42 & -0.92\end{pmatrix}$ and $A_2=\begin{pmatrix}0.65 & 0.32 \\ -0.42 & -0.92\end{pmatrix}$.
The algorithm proposed in \cite{lazar2010infinity} is employed to construct a global polytopic LF of the form \eqref{eq:polyLF}, where
\[L=\begin{pmatrix}-0.0625 & 0.6815 & 0.9947 & 0.9947 \\ 1 & 1 & 0.6868 & -0.0678\end{pmatrix}^\top,\]
and $\rho=0.94$.  We chose $\Gamma_\mx=10$ and $\Gamma_\mt=5.063$.  (see Fig.~\ref{fig:levelsets} for polytopes $\mx$, $\mt$, and a set of polytopes $\mr$.)
\end{example}

The semantics of the system can be formalized through an embedding of \eqref{eq:linearDyn} into a transition system, as follows.

\begin{definition}
\label{def:embeddedTS}
Let $\mx$, $\mt$, and $\mr=\{\mr_i\}_{i\in R}$ be given.  The embedding transition system for \eqref{eq:linearDyn} is a transition system $\TS_{e}=(Q_{e}, \Act, \to_{e},\Pi,h_{e})$ where
\begin{itemize}
\item $Q_{e}=\{x\in \rn \st x\in \mx\}$;
\item $\Act$ is the same as the index set given in Eqn~\eqref{eq:linearDyn};
\item \begin{enumerate}
\item If $x\in \mx\setminus\mt$, then  $x \toAct x'$ if and only if $x'=A_{\sigma}x$, \ie $x'$ is the state at the next time-step after applying the dynamics of \eqref{eq:linearDyn} at $x$ when subsystem $\sigma$ is active;
\item If $x\in\mt$, $x \toAct x$ for all $\sigma \in \Act$ (since the target set $\mt$ is already reached, we consider the behavior of the system thereafter no longer relevant);
\end{enumerate}
\item $\Pi=R\cup \{\Pi_\mt\}$, \ie the set of observations is the set of labels of regions, plus the label $\Pi_\mt$ for $\mt$;
\item \begin{enumerate}
  \item $h_{e}(x):=i$ if and only if $x\in \mr_{i}$;
  \item $h_e(x):=\emptyset$ if and only if $x\in\mx\setminus(\mt\cup \bigcup_{i\in R}\mr_{i})$;
  \item $h_e(x):=\Pi_\mt$ if and only if $x\in \mt$.
\end{enumerate}
\end{itemize}
\end{definition}
Note that $\TS_{e}$ is deterministic and it has an infinite number of states. Moreover, $\TS_{e}$ exactly captures the system dynamics under \eqref{eq:linearDyn} in the relevant state space $\mx\setminus\mt$, since a transition of $\TS_{e}$ naturally corresponds to the evolution of the discrete-time system in one time-step.  Indeed, within $\mx \setminus \mt$, the trajectory of $\TS_{e}$ produced by an input word $\bsigma$ from a state $x\in\mx\setminus\mt$ is exactly the same as the trajectory of system~\eqref{eq:linearDyn} from $x$ under the switching sequence $\bsigma$.

The state space of $\TS_e$ (which is the working set $\mx$) can be naturally partitioned as
\be
\label{eq:partition}
P_\mx:=\left\{\{\mr_i\}_{i\in R}, \mx\setminus(\mt\cup \bigcup_{i\in R}\mr_{i}), \mt\right\}.
\ee

The relation induced from partition $P_\mx$ is observation preserving (see Sec. \ref{sec:tsandbisim}).
We now formulate the main problem considered in this paper.
\begin{problem}\label{prob:main}
Let a system \eqref{eq:linearDyn} with a polyhedral Lyapunov function of the form \eqref{eq:polyLF}, sets $\mx$, $\mt$ and $\{\mr_i\}_{i\in R}$ be given.  Compute a finite observation preserving partition $P$ such that its induced relation $\sim$ is a bisimulation of the embedding transition system $\TS_e$, and obtain the corresponding bisimulation quotient $\TS_e\ssim$.
\end{problem}

\begin{remark}
The above assumptions on the sets $\mx$, $\mt$, and $\{\mr_i\}_{i\in R}$ are made for simplicity of presentation.
The problem formulation and the approach described in the rest of the paper can be easily extended to
arbitrary positively invariant sets $\mx$ and $\mt$, \ie not obtained as the sublevel sets of~\eqref{eq:polyLF}, by considering the largest sublevel set that is included in $\mt$ and the smallest sublevel set that includes $\mx$ ($\Gamma_\mt$ and $\Gamma_\mx$ can be made arbitrarily small and  arbitrary large, respectively, so as to capture any compact relevant subset of $\rn$). Also, the set of polytopes of interest $\{\mr_i\}_{i\in R}$ can be relaxed to a finite set of linear predicates in $x$.
\end{remark}

\section{GENERATING THE BISIMULATION QUOTIENT}
\label{sec:genBis}

Starting from a polyhedral Lyapunov function $V(x)=\|L x\|_{\infty}$ with a contraction rate $\rho=(0,1)$ as described in Sec. \ref{sec:sec:lyapunovfun} for system \eqref{eq:linearDyn}, we first generate a sequence of polytopic sublevel sets of the form $\cP_{\Gamma}:=\{x\in \rn \st \|Lx\|_{\infty}\leq \Gamma\}$ as follows.  Recall that $\mx=\cP_{\Gamma_\mx}$ and $\mt=\cP_{\Gamma_\mt}$ for some $0<\Gamma_\mt<\Gamma_\mx$.  We define a finite sequence ${\bar \Gamma}:=\Gamma_{0},\ldots,\Gamma_{N}$, where
\be
\label{eq:gammaseq}
\Gamma_{i+1}=\rho^{-1} \Gamma_{i}, \hspace{3mm} i=0,\ldots,N-2,
\ee
$\Gamma_0:=\Gamma_\mt$, $\Gamma_{N}:= \Gamma_{\mx}$, and $N:=\arg\min_N\{\rho^{-N}\Gamma_0\mid \rho^{-N}\Gamma_0\geq \Gamma_\mx\}$.  This choice of $N$ guarantees that $\cP_{\Gamma_{N-1}}$ is the largest sublevel set defined via \eqref{eq:gammaseq} that is a subset of $\mx$.
Since $\Gamma_{N}$ is  exactly $\Gamma_{\mx}$, $\cP_{\Gamma_{N}}$ is exactly $\mx$.

The sequence $\bar\Gamma$ generates a sequence of sublevel sets $\bar{\cP}_\Gamma:=\cP_{\Gamma_0},\ldots, \cP_{\Gamma_N}$. From the definition of the sublevel sets and $\bar\Gamma$, we have that
\be
\cP_{\Gamma_{0}}\subset\ldots \subset \cP_{\Gamma_{N}}.
\ee

Next, we define a \emph{slice} of the state space as follows:
\be
\label{eq:slices}
\cS_{i}:=\cP_{\Gamma_{i}}\setminus\cP_{\Gamma_{i-1}}, \hspace{3mm} i=1,\ldots,N.
\ee
For convenience, we also denote $\cS_{0}:=\cP_{\Gamma_0}$ (although $\cS_{0}$ is not a slice in between two sublevel sets).  We immediately see that the sets $\{\cS_{i}\}_{i=0,\ldots,N}$ form a partition of $\mx$.  Note that the slices are bounded semi-linear sets (see Sec. \ref{sec:prelim}).

\begin{example}[Example \ref{ex:simpleEx} continued]
\label{ex:step2}
Consider the system given in Example \ref{ex:simpleEx} and $N=11$ in Eqn.~\eqref{eq:gammaseq}. 
The polytopic sublevel sets $\bar{\cP}_\Gamma:=\cP_{\Gamma_0},\ldots, \cP_{\Gamma_{11}}$ are shown in in Fig. \ref{fig:levelsets}.
\begin{figure}[h]
   \center
   \includegraphics[scale=\myscale]{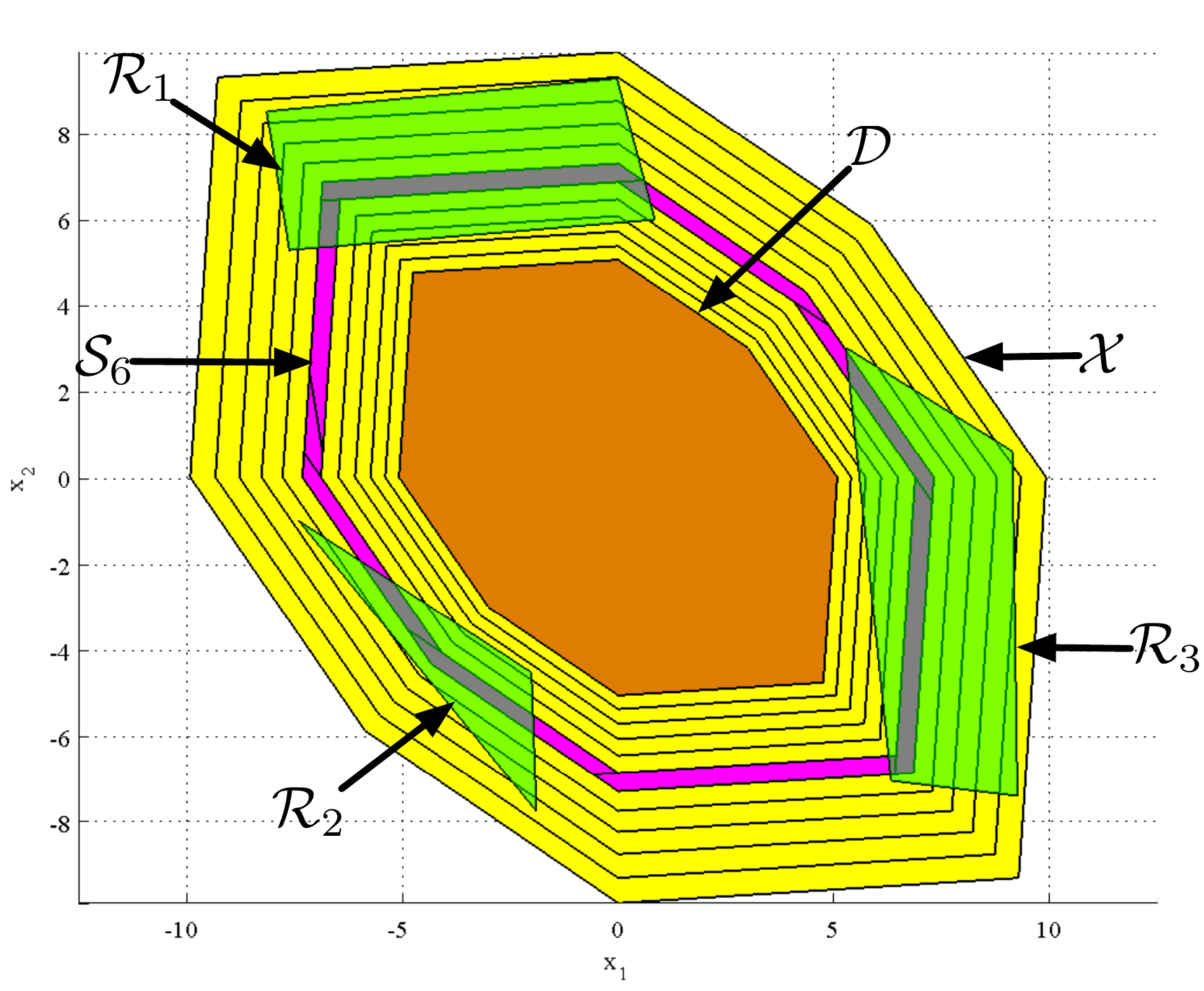}
   \caption{An example in $\Rset^2$ of the working set $\mx$, the target set $\mt$ (in brown), a set of observational relevant polytopes $\mr=\{\mr_1, \mr_2, \mr_{3}\}$ (in transparent green), sublevel sets with $N=11$ and one slice $\cS_{6}$ (in purple).}
   \label{fig:levelsets}
\end{figure}
\end{example}
\begin{proposition}
\label{prop:slicetrans}
Assume that the set of slices $\{\cS_{i}\}_{i=0,\ldots,N}$ is obtained from a sequence ${\bar\Gamma}$ satisfying \eqref{eq:gammaseq}.  Given a state $x$ in the $i$-th slice, \ie $x\in \cS_{i}$, where $1\leq i\leq N$, its successor state ($x'=A_{\sigma}x$, $\sigma \in \Sigma$) satisfies $x'\in\cS_{j}$ for some $j<i$.
\end{proposition}
\begin{proof}
From Prop. \ref{prop:rhocontractive}, we have that $\cP_{\Gamma_{i}}$ are $\rho$-contractive.  By the definition of a $\rho$-contractive set (Def. \ref{defPI}), we have that $x'=A_{\sigma}x\in \rho \cP_{\Gamma_i}=\{x\in \rn\st \|Lx\|_{\infty}\leq \rho\Gamma_{i}\}$ for all $\sigma \in \Sigma$.   From \eqref{eq:gammaseq}, we have $\rho\Gamma_{i}=\Gamma_{i-1}$.  Therefore $\cP_{\Gamma_{i-1}}=\{x\in \rn\st \|Lx\|_{\infty}\leq \Gamma_{i-1}\}$ implies that $\cP_{\Gamma_{i-1}}=\{x\in \rn\st \|Lx\|_{\infty}\leq \rho\Gamma_{i}\}$ and hence  $\cP_{\Gamma_{i-1}}=\rho \cP_{\Gamma_{i}}$ and $x'\in\cP_{\Gamma_{i-1}}$.  From the definition of slices \eqref{eq:slices}, $x'\in \cS_{j}$ for some $j<i$.
\end{proof}

We now present the abstraction algorithm (see Alg.~\ref{alg:main}) that computes the bisimulation quotient.   In Alg.~\ref{alg:main}, we make use of two procedures $\computepre$ and $\mathtt{RefineUpdate}$, which will be further explained below.  The main idea is to start with $P_{\mx}$ (Eqn.~\eqref{eq:partition}), refine the partition according to $\{\cS_{i}\}_{i=0,\ldots,N}$ to guarantee that it is a refinement to both $P_{\mx}$ as in \eqref{eq:partition} and $\{\cS_{i}\}_{i=0,\ldots,N}$, and then iteratively refine according to the $\pre$ operator (see Eqn.~\ref{eq:presigma}).  The first step, starting with $P_{\mx}$, is necessary so that the partition is observation preserving.  The second step guarantees that each element in the partition is included in a slice. The third step allows us to ensure that at iteration $i$ of the algorithm, the bisimulation quotient for states within $\cP_{\Gamma_{i}}$ is completed. 

\begin{algorithm}[h]
\caption{Abstraction algorithm}
\begin{algorithmic}[1]
\label{alg:main}
\REQUIRE System dynamics \eqref{eq:linearDyn}, polytopic LF $V(x)=\|Lx\|_{\infty}$ with a contractive rate $\rho$, sets $\mx$, $\mt$ and $\{\mr_{i}\}_{i\in R}$.
\ENSURE $\TS_{e}\ssim$ as a bisimulation quotient of the embedding transition system $\TS_{e}$ and the corresponding observation preserving partition $P$.
\STATE Obtain $P_{\mx}$ as in \eqref{eq:partition}.
\STATE Generate the sequence of sublevel sets $\bar\cP_{\Gamma} = \cP_{\Gamma_{0}},\ldots, \cP_{\Gamma_{N}}$ and slices $\cS_{0},\ldots, \cS_{N}$ as defined in\eqref{eq:slices}.
\STATE Set $P_0 = \{ \emptyset \subset \tcp_1 \cap \tcp_2 \st \tcp_1 \in P_{\mx},  \tcp_2 \in \{\cS_{i}\}_{i=0,\ldots,N} \}$.
\vspace{-12pt}
\STATE Initialize $\TS_{e}/_{\sim_{0}}$ by setting $Q_{e}/_{\sim_{0}}$ as the set labeling $P_{0}$. Set transitions only for the state $q\in Q_{e}/_{\sim_{0}}$ where $\eq (q)=\cS_{0}=\mt$ with $q\toActQ_{\sim_{0}} q$ for all $\sigma \in \Sigma$.
\FOR{each $i=0,\ldots,N-1$}
\STATE Set $\TS_{e}/_{\sim_{i+1}} = \TS_{e}/_{\sim_{i}}$ \AND $P_{i+1} = P_{i}$.
\FOR{each $q \in Q_{e}/_{\sim_{i}}$ where $\eq(q) \subseteq \cS_{i}$}
\FOR{each $\sigma \in \Sigma$}
\STATE Find $\pset=\computepre(\eq(q), \sigma)$. \label{algo:computepre}
\STATE Set $[P_{i+1}, \TS_{e}/_{\sim_{i+1}}]=\mathtt{RefineUpdate}$\\$(P_{i+1}, \TS_{e}/_{\sim_{i+1}}, \pset, \sigma,q)$. \label{algo:refineupdatestep} 
\ENDFOR
\ENDFOR
\ENDFOR
\STATE Return $\TS_{e}/_{\sim_{N}}$ and $P_{N}$ as a solution to Prob. \ref{prob:main}.
\end{algorithmic}
\end{algorithm}

The procedure $\computepre(\tcp, \sigma)$ takes as input $\tcp$, which is a bounded semi-linear set (\eg a slice), and $\sigma\in \Sigma$, which is the switching input, and returns the set $\pre_{\TS_{e}}(\tilde \cP, \sigma)$.
If $\tilde \cP$ is a polytope, then $\pre_{\TS_{e}}(\tcp, \sigma)$ is computed as
\be
\label{eq:prepoly}
\pre_{\TS_{e}}(\tcp, \sigma) = \{x\in\rn \st H_{\tcp} A_\sigma x \leq h_{\tcp}\}.
\ee

In general, if $\tcp$ is a semi-linear set, then $\pre_{\TS_{e}}(\tilde \cP, \sigma)$ is also a semi-linear set and it can be computed via quantifier elimination \cite{bochnak1998real}. In particular, $\pre_{\TS_{e}}(\tilde \cP, \sigma)$ for a bounded semi-linear set $\tcp$ can be computed via a convex decomposition and repeated applications of \eqref{eq:prepoly}.  This computation is discussed in more detail in \cite{ADHS2012_LinearSysBisim}. Note that $\computepre(\tcp, \sigma)$ only requires polytopic operations.

The procedure $\mathtt{RefineUpdate}(P, \TS, \tcp, \sigma, q)$ (outlined in Alg.~\ref{alg:partition}) refines a partition $P$ with respect to set $\tcp$, where $\pset=\computepre(\eq(q), \sigma)$.  It then updates $\TS$.  If $P$ consists of only bounded semi-linear sets and $\tcp$ is a semi-linear set, then the resulting refinement $P^{+}$ consists of only bounded semi-linear sets.  This fact allows us to always use $\computepre(\tcp, \sigma)$.

\begin{algorithm}[h]
\caption{$[P^{+}, \TS^{+}]=\mathtt{RefineUpdate}(P, \TS, \tcp, \sigma, q)$}
\begin{algorithmic}[1]
\label{alg:partition}
\REQUIRE A TS $\TS=(Q,\Act,\to,\Pi,h)$, a partition $P$ where $\eq(q') \in P$ for all $q'\in Q$, and $\pset=\computepre(\eq(q), \sigma)$ for some $q\in Q$, $\sigma\in\Sigma$.
\ENSURE $P^{+}$ is a finite refinement of $P$ with respect to $\pset$, $\TS^{+}$ is a TS updated from $\TS$.
\STATE Set $P^{+}=P$ \AND $\TS^{+} = \TS$.
\FORALL{$q' \in Q^{+}$ such that $\eq(q')\cap\tcp\neq \emptyset$}
\STATE Replace $q'$ in $Q^{+}$ by $\{q_1,q_2\}$ and set $\eq(q_1)=\eq(q')\cap\tcp$, $\eq(q_2)=\eq(q')\setminus\tcp$.
\STATE Replace $\eq(q')$ in $P^{+}$ by $\{\eq(q_1),\eq(q_2)\}$.
\STATE Replace each $(q',\sigma',q'')\in\to^{+}$ by $\{(q_i,\sigma',q'')\}_{i=1,2}$. \label{algo:transitionInherit}
\vspace{-13pt}
\STATE Add transition $(q_1,\sigma,q)$ to  $\to^{+}$. \label{algo:transitionAdd}
\ENDFOR
\end{algorithmic}
\end{algorithm}

The correctness of Alg. \ref{alg:main} will be shown by an inductive argument.
Given a sublevel set $\cP_{\Gamma_{i}}$ and a partition $P_{i}$ as obtained in Alg. \ref{alg:main}, we define $\tilde P_{i}$ as
\be
\label{eq:subpartition}
\tilde P_{i}:=\{\pset \in P_{i} \st \pset \subseteq \cP_{\Gamma_{i}}\}.
\ee
From Alg.~\ref{alg:main}, we see that $P_{0}$ partitions all the slices, and since $P_{i}$ is a finite refinement of $P_{0}$, we can directly see that $\tilde P_{i}$ is a partition of $\cP_{\Gamma_{i}}$. Let us define an embedding transition system $\TS_{e}(i)$ as a subset of $\TS_{e}$ with set of states $\{x\in Q_{e} \st x\in \cP_{\Gamma_{i}}\}$ and let us state the following result.
\begin{proposition}
\label{prop:inductionstep}
At the completion of the $i$-th iteration (of the outer loop) of Alg.~\ref{alg:main} (where $P_{i+1}$ is obtained), if $\sim_{i}$ induced by $\tilde P_{i}$ as defined in \eqref{eq:subpartition} is a bisimulation of $\TS_{e}(i)$, then $\sim_{i+1}$ induced by $\tilde P_{i+1}$ is a bisimulation of $\TS_{e}(i+1)$.
\end{proposition}
\begin{proof}
We show that at the end of $i$-th iteration, each transition originating at a state $q \in Q_{e}/_{\sim_{i+1}}$ with $\eq(q) \subseteq \cP_{\Gamma_{i+1}}$ satisfies the bisimulation requirement (Def.~\ref{def:bisimulation}).
By Prop.~\ref{prop:slicetrans}, for each $x\in \cS_{i+1}$ and $\sigma \in \Sigma$, $x'=A_\sigma x$ must be in a slice with a lower index and thus $x'\in\TS_{e}(i)$.  Let $x \in \eq(q) \in P_{i}$. If $x'\in \cS_{i}$, then we have $x\in \pset=\computepre(\eq(q'), \sigma)$ (from step~\ref{algo:computepre} of Alg. \ref{alg:main}) for some $q' \in Q_{e}/_{\sim_{i}}$.
The $\mathtt{RefineUpdate}$ procedure replaces $\eq(q)$ with $\eq(q_1) = \eq(q) \cap \pset$ and $\eq(q_2)=  \eq(q) \setminus \pset$, and updates $\TS_{e}/_{\sim_{i+1}}$.
We note from Eqn.~\eqref{eq:presigma} that for any $x\in \eq(q_1)$, $x'=A_\sigma x\in \eq(q')$, thus for any $x_1, x_2 \in\eq(q_1)$, $x_{1}\sim x_{2}$, $A_\sigma x_{1}\sim A_\sigma x_{2}$. Moreover, the same argument holds for any subset of $\eq(q_1)$. Therefore, the transitions given in steps~\ref{algo:transitionInherit} and~\ref{algo:transitionAdd} of Alg.~\ref{alg:partition} satisfy the bisimulation requirement.
On the other hand, if $x'\notin S_{i}$, then $x'\in S_{j}$ for some $j<i$ and $x$ is already in a set $\eq(q)$, where $q \toActQ_{\sim_{i+1}} \eq(q')$ for some $q'$ satisfying the bisimulation requirement.  Therefore, step $9$ of Alg.~\ref{alg:main} provides exactly the transitions needed for all states in $\cS_{i+1}$ and thus, $\sim_{i+1}$ induced by $\tilde P_{i+1}$ is a bisimulation of $\TS_{e}(i+1)$.
\end{proof} 

\begin{theorem}
Alg. \ref{alg:main} returns a solution to Prob.~\ref{prob:main} in finite time.
\end{theorem}
\begin{proof}
From Alg.~\ref{alg:partition}, we have that $P_{i}$ is a refinement of $P_{\mx}$ for any $i=0,\ldots,N$.  Therefore, $P_{N}$ and its induced relation $\sim_{N}$ are observation preserving.

At step $4$ of Alg. \ref{alg:main}, we set $q\toActQ_{\sim_{0}} q, \forall \sigma \in \Sigma$ where $\eq(q)=\mt$.  From the definition of $\TS_{e}$, we see that since $\mt$ is the only state, $\sim_{0}$ induced by $\tilde P_{0}$ is a bisimulation of $\TS_{e}(0)$.  Using Prop. \ref{prop:inductionstep} and induction, at iteration $N-1$, we have that $\sim_{N}$ induced by $\tilde P_{N}$ is a bisimulation of $\TS_{e}(N)$.   Note that $\tilde P_{N}$ is exactly $P_{N}$, $\cP_{\Gamma_{N}}$ is exactly $\mx$ and $\TS_{e}(N)$ is exactly $\TS_{e}$.  Therefore $\sim_{N}$ induced by $P_{N}$ is a bisimulation of $\TS_{e}$.

Finally, note that at each iteration $i$, the number of updated sets is finite as the partition $P_i$ and the set of inputs $\Sigma$ are finite.  Therefore, the bisimulation quotient is finite and Alg.~\ref{alg:main} completes in finite time.
\end{proof}

\begin{figure*}[ht]
	\centering
	\subfloat[]{\includegraphics[scale=\myscaleex]{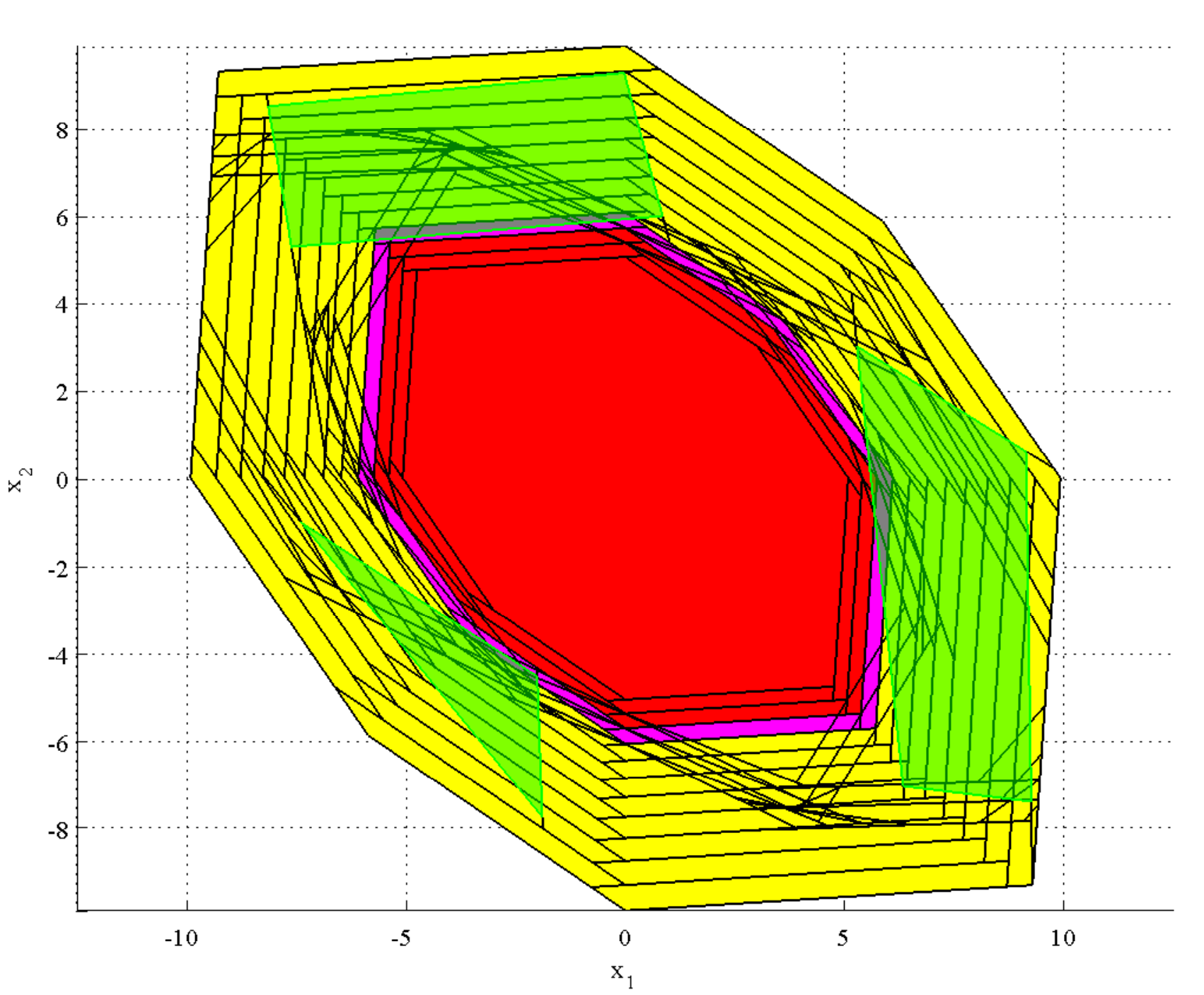} \label{fig:it2e}}
	\subfloat[]{\includegraphics[scale=\myscaleex]{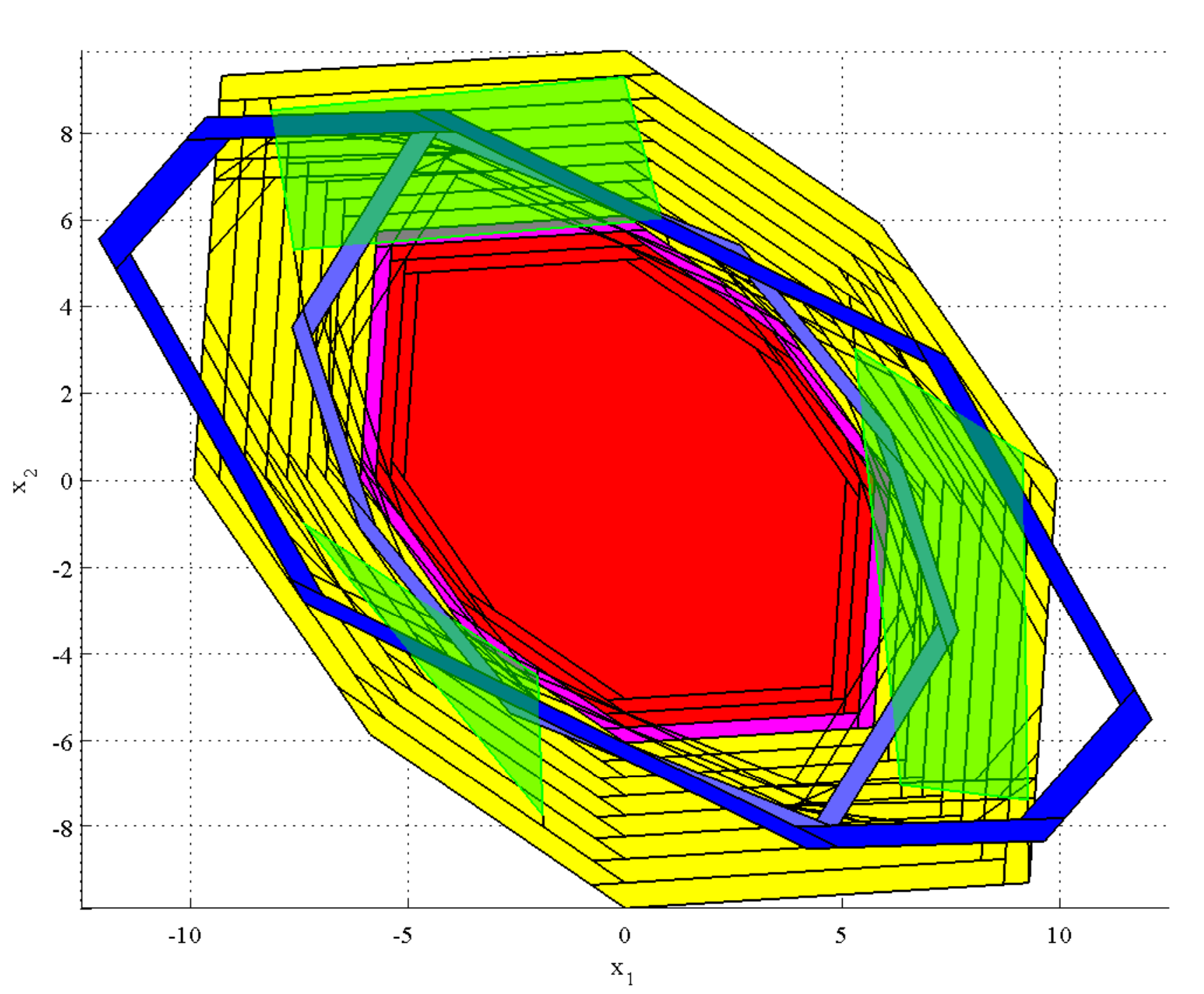}  \label{fig:it3s}}
	\subfloat[]{\includegraphics[scale=\myscaleex]{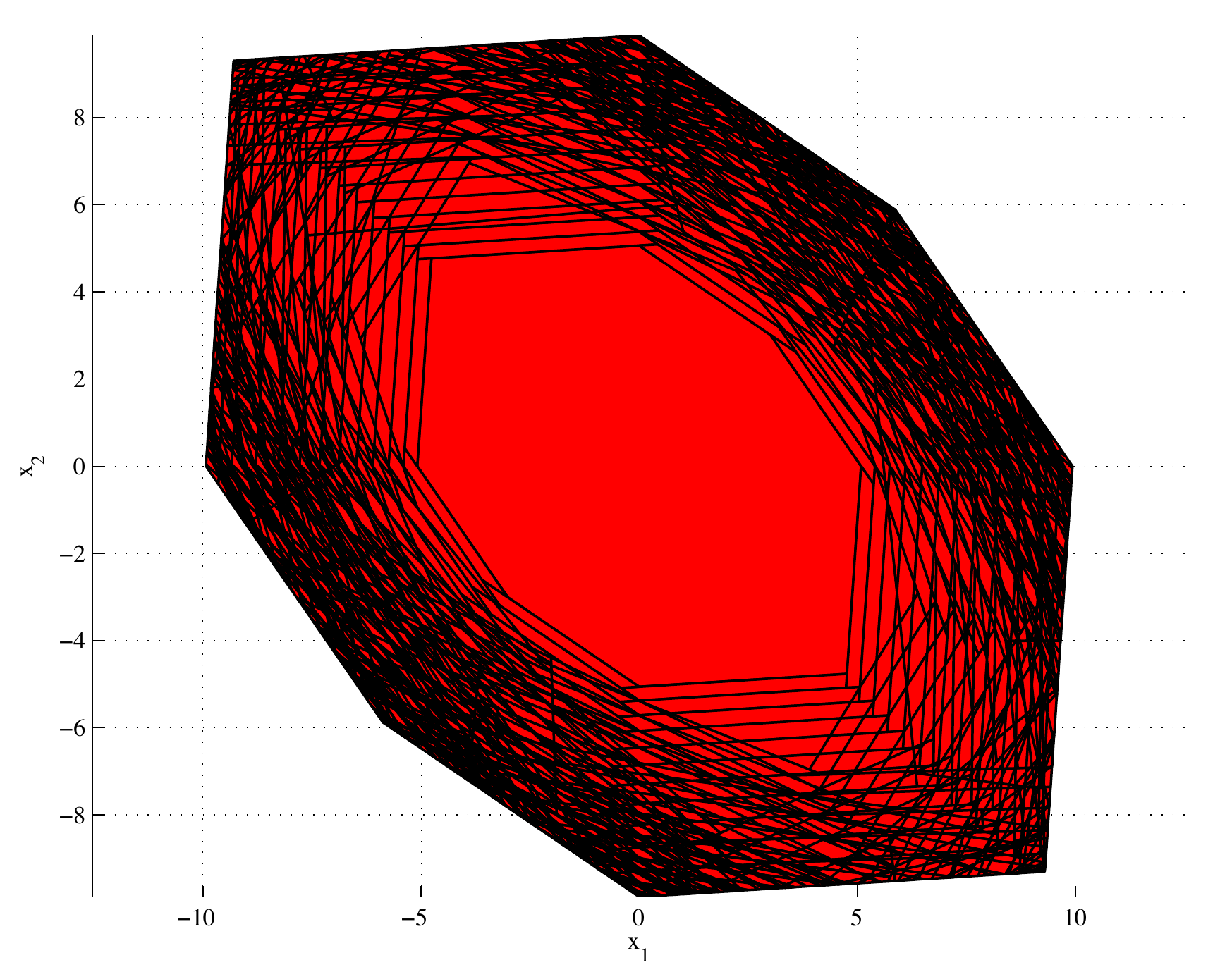}\label{fig:itend}}
\caption{The observed regions are shown in transparent green in (a) and (b). (a) At the end of the third iteration ($i=2$), the bisimulation quotient for states within $\cP_{\Gamma_{3}}$ is completed, which are shown in red and purple. In the forth iteration, the states within $\cP_{\Gamma_{11}} \setminus \cP_{\Gamma_{3}}$ will be partitioned according to $\pre_{\TS_{e}}(\tcp, \sigma), \pset \in \cS_{3}$. (b) $\cS_{3}$ is shown in purple, and $\pre_{\TS_{e}}(\cS_{3}, 1)$ and $\pre_{\TS_{e}}(\cS_{3},2)$ are shown in light and dark blue. (c) At the last iteration where $i=10$, the algorithm is completed.  The state space covered by the bisimulation quotient is shown in red, covering all of $\mx$.}
\label{fig:snapshots}
\end{figure*}

\begin{example}[Example \ref{ex:step2} continued]
 \label{ex:step3}
Alg.~\ref{alg:main} is applied on the same setting as in Example~\ref{ex:step2} to compute the bisimulation quotient.
$P_3$ and $P_{11}$ are shown in Fig.~\ref{fig:snapshots}.
\end{example}

\section{TEMPORAL LOGIC SYNTHESIS AND VERIFICATION}
\label{sec:verification}
After we obtain a bisimulation quotient for system~\eqref{eq:linearDyn}, we can solve verification and controller synthesis problems from temporal logic specifications such as CTL*, CTL and LTL.
The asymptotic stability assumption implies that all trajectories of~\eqref{eq:linearDyn} sink in $\mt$.
For this reason, we will focus on syntactically co-safe fragment of LTL, which includes all specifications of LTL where satisfactions of trajectories can be determined by a finite prefix. Since we are interested in the behavior of~\eqref{eq:linearDyn} until $\mt$ is reached, scLTL is sufficiently rich as the specification language.

A detailed description of the syntax and semantics of scLTL is beyond the scope of this paper and can be found in, for example, \cite{Vardi:safety, Clarke99}.  Roughly, an scLTL formula is built up from a set of atomic propositions $\Pi$, standard Boolean operators $\neg$ (negation), $\LTLor$ (disjunction), $\LTLand$ (conjunction), $\Rightarrow$ (implication) and temporal operators $\LTLNext$ (next), $\LTLUntil$ (until) and $\LTLEvent$ (eventually). The semantics of scLTL formulas is given over infinite words $\bo=o_{0}o_{1}\ldots$, where $o_{i} \in 2^{\Pi}$ for all $i$.  We write $\bo\vDash\phi$ if the word $\bo$ satisfies the scLTL formula $\phi$.  We say a trajectory $\bq$ of a transition system $\TS$ satisfies scLTL formula $\phi$, if the word generated by $\bq$ (see Def. \ref{def:tran_sys}) satisfies $\phi$. 

\begin{example}\label{ex:LTLexample}
Again, consider the setting in Example \ref{ex:simpleEx} with $\mr=\{\mr_i\}_{i=\{1,2,3\}}$.  We now consider a specification in scLTL over $\{\mr_{1},\mr_{2},\mr_{3}, \Pi_\mt\}$.  For example, the specification \emph{``A system trajectory never visits $\mathcal R_{2}$ and eventually visits $\mr_{1}$.  Moreover, if it visits $\mr_{3}$ then it must not visit $\mr_{1}$ at the next time step''} can be translated to a scLTL formula:
\begin{equation}
\label{eq:exampleLTLformula}
\phi:= ( \neg \mr_{2} \LTLUntil \Pi_ \mt) \LTLand \LTLEvent \mr_{1} \LTLand ((\mr_{3} \Rightarrow \LTLNext \neg \mr_{1} ) \LTLUntil  \Pi_\mt)
\end{equation}
\end{example}

\subsection{Synthesis of switching strategies}

In this section, we assume that we can choose the dynamics $A_\sigma$, $\sigma\in\Sigma$ to be applied at each step $k$. Our goal is to find a set of initial states and a switching sequence (i.e., a sequence of elements from $\Sigma$ to be applied at each step)
for each initial state such that all the corresponding trajectories of system \eqref{eq:linearDyn} satisfy a temporal logic specification. Formally, we consider the following problem:

\begin{problem}
\label{prob:LTL}
Consider system \eqref{eq:linearDyn} with a polyhedral Lyapunov function in the form of \eqref{eq:polyLF}, sets $\mx$, $\mt$ and $\{\mr_i\}_{i\in R}$, and a scLTL formula $\phi$ over $R\cup \{\Pi_\mt\}$.  
Find the largest set $\Qsat \subseteq \mx$ and a function $\Omega: \Qsat \mapsto \setI^*$ such that
the trajectory of system~\eqref{eq:linearDyn} initiated from a state $x_0 \in \Qsat$ under the switching sequence $\Omega(x_0)$ satisfies $\phi$.
\end{problem}

As a switched system is deterministic, it produces a unique trajectory for a given initial state and switching sequence.
This fact allows us to provide a solution to Problem \ref{prob:LTL} as an assignment of a switching sequence to each initial state. Our solution to Prob. \ref{prob:LTL} proceeds by finding a bisimulation quotient $\TS_{e}/_\sim$ of the embedding transition system $\TS_e$ using Alg. \ref{alg:main}.  Then we translate $\phi$ to a Finite State Automaton (FSA), defined below.

\begin{definition} A deterministic finite state automaton (FSA) is a tuple $\FSA = (S_\FSA, S_{\FSA0}, \Sigma, \delta_\FSA, F_\FSA)$ where
\begin{itemize}
	\item $S_\FSA$ is a finite set of states;
	\item $S_{\FSA0} \subseteq S_\FSA$ is a set of initial states;
	\item $\Sigma$ is an input alpabet;
	\item $\delta_\FSA: S_\FSA \times \Sigma \rightarrow S_\FSA$ is a transition function;
	\item $F_\FSA \subseteq S_\FSA$ is a set of final states.
\end{itemize}
\end{definition}
A word $\bsigma=\sigma_0\ldots \sigma_{d-1}$ over $\Sigma$ generates a trajectory $s_0\ldots s_{d}$, where $s_0 \in S_{\FSA0}$ and $\delta(s_i,\sigma_i) = s_{i+1}$ for all $i=0,\ldots,d-1$. $\FSA$ accepts word $\bsigma$ if $s_{d} \in F_{\FSA}$.

For any scLTL formula $\phi$ over $\Pi$, there exists a FSA $\FSA$ with input alphabet $2^\Pi$ that accepts the prefixes of all and only the satisfying words \cite{Vardi:safety,Latvala:scheck}.

\begin{definition}
\label{def:PA}
 Given a transition system $\TS=(Q,\Sigma, \to,\Pi,h)$ and a FSA $\FSA = (S_\FSA, S_{\FSA0}, 2^{\Pi}, \delta_\FSA, F_\FSA)$ , their product automaton, denoted by $\Prod=\TS \times \FSA$, is a tuple $\Prod=(S_\Prod,S_{\Prod0}, \Sigma,\to_\Prod,F_\Prod)$
where
\begin{itemize}
 \item $S_\Prod = Q \times S_{\FSA}$;
 \item $S_{\Prod0} = Q \times S_{\FSA 0}$;
 \item $\to_{\Prod}\subseteq S_{\Prod}\times \Sigma \times S_{\Prod}$ is the set of transitions, defined by: $\left((q,s),\sigma,(q',s')\right)\in \to_{\Prod}$ iff $q \toActQ q'$ and $\delta_{\FSA}(s, h(q)) = s'$;
 \item $F_\Prod = Q \times F_{\FSA}$.
\end{itemize}
We denote $s_{\Prod} \toActQ_{\Prod} s'_{\Prod}$ if $(s_{\Prod}, \sigma, s'_{\Prod})\in \to_{\Prod}$.
A trajectory ${\rm \bp}=(q_{0}, s_{0})\ldots (q_{d}, s_{d})$ of $\Prod$ produced by input word $\bsigma=\sigma_0\ldots \sigma_{d-1}$ is a finite sequence such that $(q_{0}, s_{0})\in S_{\Prod 0}$ and $(q_{k},s_{k}) \stackrel{\sigma_k}{\to}_{\Prod} (q_{k+1},s_{k+1})$ for all $k=0,\ldots,d-1$.  $\bp$ is called accepting if $(q_{d}, s_{d}) \in F_\Prod$.
\end{definition}

By the construction of $\Prod$ from $\TS$ and $\FSA$, $\bp$ produced by $\bsigma$ is accepting if and only if $\bq=\gamma_{\TS}(\bp)$ satisfies the scLTL formula corresponding to $\FSA$ \cite{Clarke99}, where $\gamma_{\TS}(\bp)$ is the projection of a trajectory $\bp$ of $\Prod$ onto $\TS$ by simply removing the automaton part of the state in $s_{\Prod}\in S_{\Prod}$.

We construct the product $\Prod$ between the quotient transition system $\TS_{e}/_\sim$ obtained from Alg.~\ref{alg:main} and FSA $\FSA$ corresponding to specification formula $\phi$. By performing a graph search on $\Prod$, we can find the largest subset $S_{\Prod}^S$ of $S_{\Prod}$ and a feedback control function $\Omega_\Prod : S^S_\Prod \mapsto \Sigma$ such that the trajectories of $\Prod$ originating in $S_\Prod^S$ in closed loop with $\Omega_\Prod$ reach $F_{\Prod}$.
Then, we define the set of satisfying initial states of system~\eqref{eq:linearDyn} from $S_{\Prod}^S$ as
\be \label{eq:XS}
   \Qsat = \{ \eq(q) \mid (q, s) \in  (S_{\Prod0} \cap S^S_{\Prod}) \}.
\ee

Since $\Prod$ is deterministic, $\Omega_\Prod$ defines a unique input word for each $(q_0,s_0) \in S^S_\Prod$. Moreover, an input word of $\Prod$ directly maps to a switching sequence for system~\eqref{eq:linearDyn}. Formally, the switching sequence $\Omega: \Qsat \mapsto \Sigma^*$ is obtained by ``projecting'' $\Omega_\Prod$ from $\Prod$ to $\TS$ as follows:
\be \label{eq:openloop}
   \Omega(x) = \Omega_{\Prod}((q_0,s_0))\ldots\Omega_\Prod((q_{d-1},s_{d-1})) ,
\ee
where $x\in eq(q_0)$, $s_0 \in S_{\FSA 0}$, $(q_i,s_i)\overset{\Omega_{\Prod}((q_i,s_i))}{\xrightarrow{\hspace*{1.8cm}}}_{\Prod} (q_{i+1},s_{i+1})$, for each $i=0,\ldots,d-1$ and  $(q_d,s_d) \in F_{\Prod}$.

\begin{proposition}
$\Qsat$ as defined in Eqn.~\eqref{eq:XS} and function $\Omega$ as defined in Eqn.~\eqref{eq:openloop} solve Prob.~\ref{prob:LTL}.
\end{proposition}
\begin{proof}
For each $x\in \Qsat$, there exists $(q_0,s_0) \in S^S_\Prod$ such that $x\in \eq(q_0)$ and $s_0 \in S_{\FSA 0}$ by Eqn.~\eqref{eq:XS}. By construction of $\Prod$ and definition of $\Omega$ (Eqn.~\eqref{eq:openloop}), the trajectory of $\TS_{e}/_\sim$ originating at $q_0$ and generated by input word $\Omega(x)$ satisfies $\phi$. Then by bisimulation relation the trajectories of~\eqref{eq:linearDyn} originating in $\eq(q_0)$ and generated by switching sequence $\Omega(x)$ satisfy $\phi$.

We prove that $\Qsat$ is the largest set of satisfying initial states by contradiction.
Assume that there exists $x_{0}\notin \Qsat$ such that a trajectory $\bx =x_{0}\ldots x_{d}$ originating at $x_{0}$ of~\eqref{eq:linearDyn} produced by switching sequence $\bsigma=\sigma_0\ldots \sigma_{d-1}$ satisfies $\phi$, and $x_{0}\in \eq(q_0)$ where $q_0~\in~Q_{e}\ssim$. Then by the bisimulation relation 1) there exists a trajectory $\bq =q_{0}\ldots q_{d}$ of $\TS_{e}\ssim$ such that $x_i \in \eq(q_i)$, $q_i \stackrel{\sigma_i}{\to}_{e}\ssim q_{i+1}$ for all $i=0,\ldots,d-1$ and $x_{d} \in \eq(q_d)$, 2) $\bq$ satisfies $\phi$. 
However, we know that on the product $\Prod = \TS_{e}\ssim\times \FSA$, $F_{\Prod}$ is not reachable from $\{ (q_0, s) \st s \in S_{\FSA 0} \}$. 
Hence, a trajectory $\bp$ originating in $\{ (q_0, s) \st s \in S_{\FSA 0} \}$ cannot be accepting on $\Prod$, and by construction of $\Prod$ \cite{Clarke99} $\gamma_{\TS_{e}\ssim}(\bp)$ as a trajectory of $\TS_{e}\ssim$ cannot satisfy formula $\phi$, which yields a contradiction. 
\end{proof}

\begin{example}[Example \ref{ex:LTLexample} continued]
For the example specification $\phi$~\eqref{eq:exampleLTLformula}, we obtained the solution to Prob. \ref{prob:LTL}. 
The FSA has 6 states and the quotient TS obtained from Alg.~\ref{alg:main} has 9677 states. 
The set of initial states $\Qsat$ is shown in Fig.~\ref{fig:ltlresult}. 
\begin{figure}[h]
   \center
   \includegraphics[scale=\myscale]{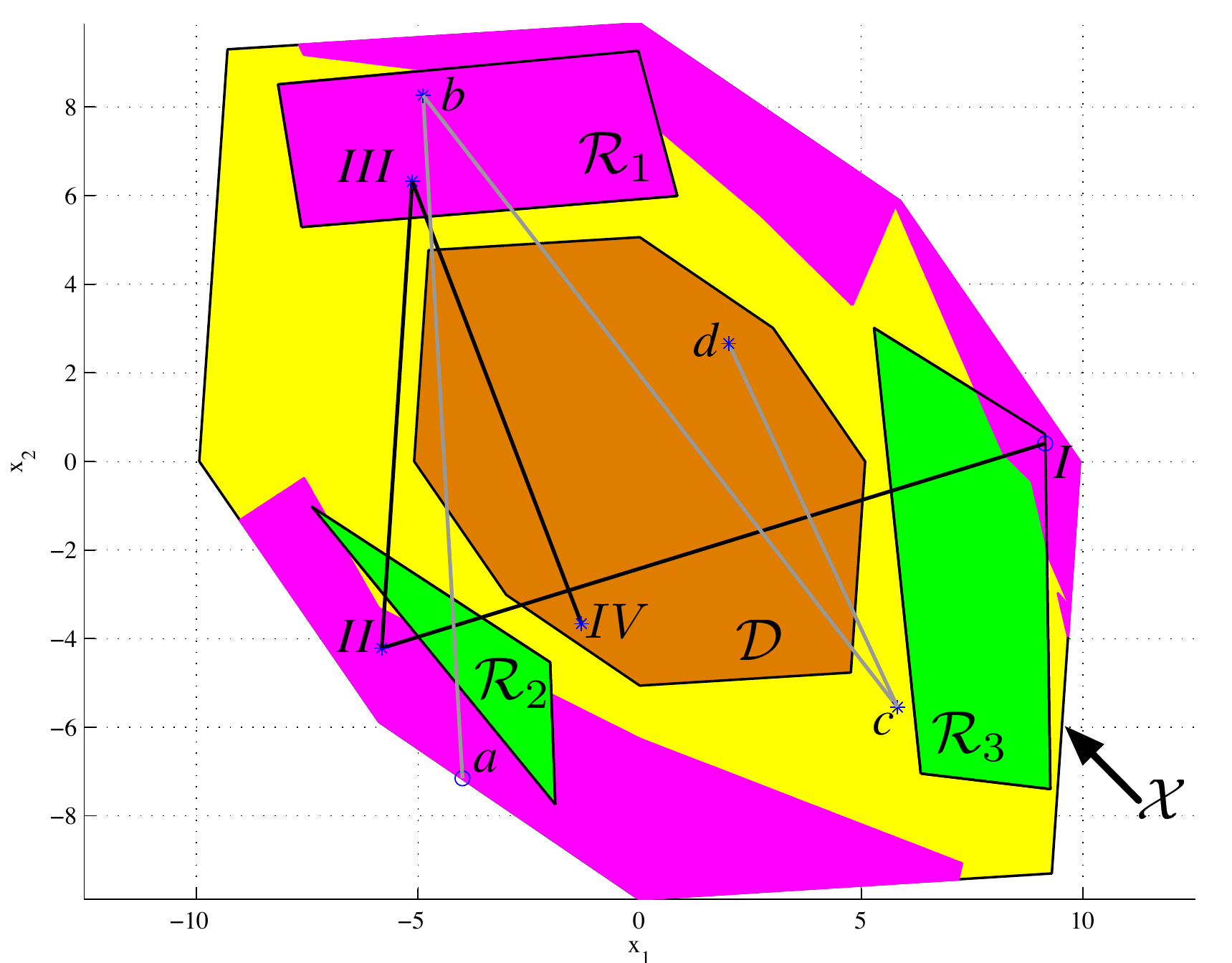}
   \caption{$\Qsat$ is shown in purple. $\mx$, $\mt$, $\{\mr_i\}_{i\in R}$ and two sample trajectories are indicated by their labels.}
   \label{fig:ltlresult}
\end{figure}
\end{example}

\newcommand{\TSEA}{\TS_{e}^A}
\newcommand{\TSQA}{\TS^A_{e}\ssim}
\newcommand{\ActAE}{\to^A_{e}}
\newcommand{\QsatA}{\mathcal{X}^{AS}}
\subsection{Verification under arbitrary switching}
\begin{problem}
\label{prob:LTLverification}
Consider system \eqref{eq:linearDyn} with a polyhedral Lyapunov function in the form of \eqref{eq:polyLF}, sets $\mx$, $\mt$ and $\{\mr_i\}_{i\in R}$, and a scLTL formula $\phi$ over $R \cup \{\Pi_\mt\}$.  
Find the largest set $\QsatA \subseteq \mx$ such that all trajectories of system~\eqref{eq:linearDyn} originating in $\QsatA$ satisfy $\phi$ under arbitrary switching.
\end{problem}

Note that system~\eqref{eq:linearDyn} under arbitrary switching is uncontrolled and non-deterministic, $\ie$ at every time-step a subsystem is arbitrarily chosen from the set $\Sigma$.
Therefore, we define an embedding transition system $\TSEA = \{Q_{e}, \Sigma^A, \to^A_{e}, h_{e}\}$ for the arbitrary switching setup from the embedding transition system $\TS_{e} = \{Q_{e}, \Sigma, \to_{e}, h_{e}\}$ (Def.~\ref{def:embeddedTS}) by adapting the input set and the set of transitions as follows:
\begin{itemize}
	\item $\Sigma^A = \{ \epsilon\}$,
	\item $\to^A_{e} = \{ (q,\epsilon, q') \st \exists \sigma \in \Sigma, (q,\sigma, q') \in \to_{e}\}$.
\end{itemize}
We denote $q \to^A_{e} q'$ if $(q,\epsilon, q') \in  \to^A_{e}$. We use $\epsilon$ as a ``dummy'' input because the transitions of $\TSEA$ are not controlled. Note that $\TSEA$ is infinite and non-deterministic.  Moreover, $\TSEA$ exactly captures dynamics of system~\eqref{eq:linearDyn} under arbitrary switching in the relevant state space $\mx\setminus\mt$.

Our solution to Prob~\ref{prob:LTLverification} parallels the solution we proposed for Prob.~\ref{prob:LTL}.
We first convert the bisimulation quotient  $\TS_{e}\ssim = \{Q_{e}\ssim, \Sigma, \to_{e}\ssim, h_{e}\ssim\}$ of $\TS_{e}$ obtained from Alg.~\ref{alg:main} to $\TSQA = \{Q_{e}\ssim, \Sigma^A, \to^A_{e}\ssim, h_{e}\ssim\}$ as follows:
\begin{itemize}
	\item $\Sigma^A = \{ \epsilon\}$,
	\item $\to^A_{e}\ssim = \{ (q,\epsilon, q') \st \exists \sigma \in \Sigma, (q,\sigma, q') \in \to_{e}\ssim\}$.
\end{itemize}
In this case, we have a particular bisimulation relation. The embedding and the quotient transition systems have a single input that labels all the transitions. 

\begin{proposition}
$\TSQA$ is a bisimulation quotient of $\TSEA$.
\end{proposition}
\begin{proof}
Let $q_1, q_2 \in \eq(q)$, $q_1' \in \eq(q')$ and $q_1 \ActAE q_1'$, where $q,q' \in Q_{e}\ssim$ and $q_1,q_2,q_1' \in Q_{e}$. To prove the bisimulation property we need to show that there exists $q_2' \in \eq(q')$ such that $q_2 \ActAE q_2'$.

If $q_1 \ActAE q_1'$, then there exists $\sigma \in \Sigma$ such that $q_1\toAct q_1'$, \ie $q_1' = A_\sigma q_1$. Steps~\ref{algo:computepre} and~\ref{algo:refineupdatestep} of Alg.~\ref{alg:main} guarantee that $\eq(q) \subseteq \pre_{\TS_{e}}(\eq(q'), \sigma)$. Therefore, for all $q_i \in \eq(q)$, $A_\sigma q_i \in \eq(q')$, and hence for all $q_i \in \eq(q)$, $q_i \ActAE q_j$ for some $q_j \in \eq(q')$.
\end{proof}

Parallel to our solution to Prob.~\ref{prob:LTL}, we construct a FSA $\FSA$ corresponding to specification formula $\phi$, and then we take the product $\Prod^A=(S^A_\Prod,S^A_{\Prod0}, \Sigma^A,\to^A_\Prod,F^A_\Prod)$ between $\TSQA$ and $\FSA$ as described in Def.~\ref{def:PA}. Note that $\Prod^A$ is non-deterministic as $\TSQA$ is non-deterministic.

To finally solve Prob. \ref{prob:LTLverification}, we formulate the fixed point problem:
\be\label{eq:fixedpoint}
J(s_{\Prod}) = \min(J(s_{\Prod}), \max_{s_{\Prod} \to^A_\Prod s'_{\Prod}}{J(s'_{\Prod}) + 1} ),
\ee
initialized with $J(s_{\Prod}) = \infty$ for all $s_{\Prod} \in S^A_\Prod \setminus F^A_\Prod$ and $J(s_{\Prod}) = 0$ for all $s_{\Prod} \in F^A_\Prod$.
\begin{proposition} Let $S^{AS}_{\Prod} = \{s_{\Prod} \in S^{A}_{\Prod} \st J(s_{\Prod}) < \infty\}$ and define $\QsatA = \{ \eq(q) \mid (q, s) \in  (S^A_{\Prod0} \cap S^{AS}_{\Prod}) \}$. Then
$\QsatA$ solves Prob.~\ref{prob:LTLverification}.
\end{proposition}
\begin{proof}
For each $x\in \QsatA$, there exists $(q_0,s_0) \in S^{AS}_\Prod$ such that $x\in \eq(q_0)$ and $s_0 \in S_{\FSA 0}$. 
The fixed point algorithm guarantees that every trajectory of $\Prod^A$ originating at $(q_0,s_0)$ reaches $F^A_\Prod$.
Then, construction of $\Prod^A$ and bisimulation relation guarantee that all of the trajectories of~\eqref{eq:linearDyn} originating in $\eq(q_0)$ satisfy $\phi$.

If $x_0 \notin \QsatA$, we need to show that there exists a trajectory $\bx = x_{0}\ldots x_{d}$ of~\eqref{eq:linearDyn} that violates $\phi$.
  Let $x_0 \in \eq(q_{0})$. If $x_0 \notin \QsatA$, then for all $s_0 \in S_{\FSA 0}$ there exists a trajectory $\bp = (q_0, s_0)\ldots(q_d, s_d)$ of $\Prod^A$ that can not reach $F^A_{\Prod}$, otherwise $(q_0, s_0)$ would be included in $S^{AS}_{\Prod}$. Since $\bp$ can not reach $F^A_{\Prod}$, $\bq=\gamma_{\TS}(\bp)$ violates $\phi$. By bisimulation property, there exists a trajectory $\bx = x_{0}\ldots x_{d}$ of~\eqref{eq:linearDyn} that produces the same word as $\bq$, and hence $\bx$ violates $\phi$.
\end{proof}

\begin{example}[Example \ref{ex:LTLexample} continued]
For the example specification $\phi$ as in \eqref{eq:exampleLTLformula}, we obtained the solution to Prob. \ref{prob:LTLverification}.  $\QsatA$ and sample trajectories are shown in~Fig. \ref{fig:ltlresultverification}.  Note that this is a subset of the set of initial states found for the synthesis problem (see Fig.\ref{fig:ltlresult}).
\begin{figure}[h]
   \center
   \includegraphics[scale=\myscale]{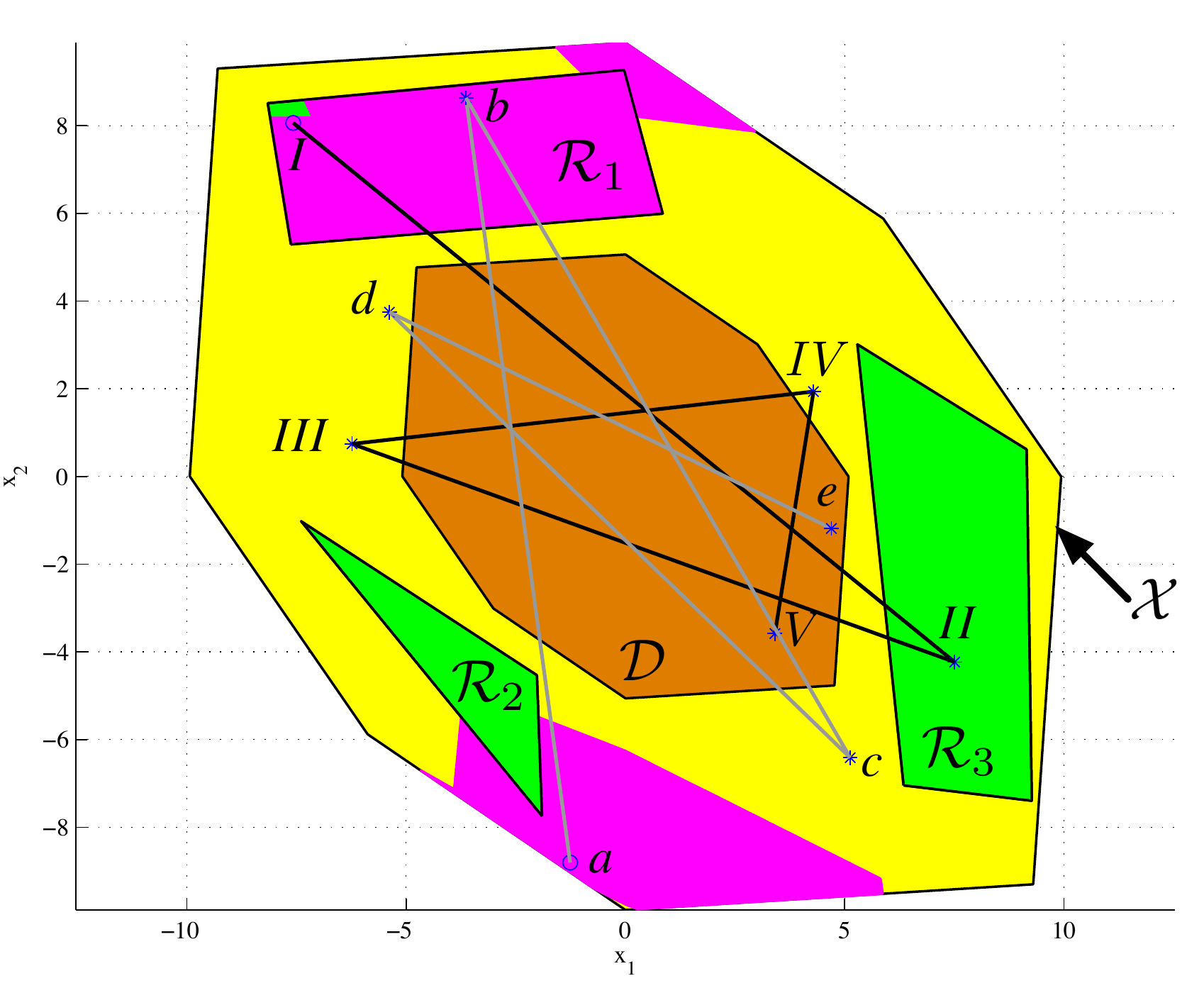}
   \caption{$\QsatA$ is shown in purple. $\mx$, $\mt$, $\{\mr_i\}_{i\in R}$ and two sample trajectories are indicated by labeling.}
   \label{fig:ltlresultverification}
\end{figure}
\end{example}

\begin{remark}[Implementation]
The methods described in this paper were implemented in MATLAB as a software package\footnote{A preliminary version of the software is downloadable from $\mathtt{hyness.bu.edu/switchedpolybis.html}$.}, which uses the MPT toolbox\cite{mpt} for polyhedral operations. 
Alg.~\ref{alg:main} was completed in 2 hours for the example presented in the paper on an iMac with a Intel Core i5 processor at 2.8GHz with 8GB of memory.
 Once the bisimulation quotient is constructed, controller synthesis and verification were both completed in 2 minutes.
 \end{remark}
\section{Conclusions}
\label{sec:concl}
In this paper, we presented a method to abstract the behavior of a switched linear system within a positively invariant subset of $\Rset^n$ to a finite transition system via the construction of a bisimulation quotient.  We employed polyhedral Lyapunov functions to guide the partitioning of the state space and showed that the construction requires polytopic operations only.
We showed how this method can be used to synthesize switching sequences and to verify the behavior of the system under arbitrary switching from specifications given as scLTL formulas over linear predicates in the state of the system.


\end{document}